\documentclass[a4paper, 9pt]{article}

\usepackage{graphicx}
\usepackage{amsmath,amsthm}
\usepackage{amsfonts,amssymb}
\usepackage{float}

\usepackage{enumerate}
\usepackage{stmaryrd}

\usepackage{hyperref}
\hypersetup{urlcolor=blue,linkcolor=black,citecolor=black,colorlinks=true}

\newtheorem{theorem}{Theorem}
\newtheorem{proposition}[theorem]{Proposition}
\newtheorem{corollary}[theorem]{Corollary}
\newtheorem{lemma}[theorem]{Lemma}
\newtheorem{assumption}[theorem]{Assumption}

\theoremstyle{definition}
\newtheorem{example}[theorem]{Example}
\newtheorem{remark}[theorem]{Remark}

\newcommand{\E}{\mathbb{E}}

\newcommand{\N}{\mathbb{N}}
\newcommand{\Q}{\mathbb{Q}}
\newcommand{\R}{\mathbb{R}}
\newcommand{\Pb}{\mathbb{P}}

\newcommand{\F}{\mathcal{F}}

\newcommand{\C}{\mathcal{C}}
\newcommand{\G}{\mathcal{G}}

\def\={{\;\mathop{=}\limits^{\text{(law)}}\;}}

\allowdisplaybreaks

\begin{document}

\begin{center}
{\LARGE \textbf{On last passage times of linear diffusions\\
\vspace{.3cm}
 to curved boundaries}}\\
\vspace{.7cm}
{\large Christophe \textsc{Profeta}\footnote{Laboratoire d'Analyse et Probabilit\'es, Universit\'e d'\'Evry - Val d'Essonne, B\^atiment I.B.G.B.I., 3\`eme \'etage, 23 Bd. de France, 91037 EVRY CEDEX.\\
 \textbf{E-mail}: christophe.profeta@univ-evry.fr}
}\\
\vspace{.8cm}
\end{center}

\textbf{Abstract:	} The aim of this paper is to study the law of the last passage time of a linear diffusion to a curved boundary. We start by giving a general expression for the density of such a random variable under some regularity assumptions. Following Robbins \& Siegmund, we then show that this expression may be computed for some implicit boundaries via a martingale method.  Finally, we discuss some links between first hitting times and last passage times via time inversion, and present an integral equation (which we solve in some particular cases) satisfied by the density of the last passage time.  Many examples are given in the Brownian and Bessel frameworks.\\


\section{Introduction}
\subsection{Motivation}
Let $\ell \in [-\infty,+\infty[$ and consider a linear regular  conservative diffusion $(X_t, t\geq0)$  taking values in $I=(\ell,+\infty[$ with $+\infty$ a natural boundary. 
\\
\\
Let $f:[0,+\infty[\longrightarrow [\ell,+\infty[$ be a continuous function and define $\zeta(f):=\inf\{t\geq0\,;\, f(t)=\ell\}\in]0,+\infty]$. The aim of this paper is to compute the law of the last passage time of $X$ to the boundary $f$ before time $\zeta(f)$: 
$$
\begin{array}{rl}
G_f&:=\sup\{0\leq t\leq \zeta(f)\,;\, X_t=f(t)\},\\
\vspace{-.2cm}\\
&\!(= 0 \quad \text{if }\{0\leq t\leq \zeta(f)\,;\, X_t=f(t)\}=\varnothing).
\end{array}
$$
This problem was essentially addressed (in a far more general framework) in the literature through the study of additive functionals. We refer in particular to  Getoor \& Sharpe \cite[Proposition 3.3]{GS} where the law of the last exit time from a Borel set $D$ of a general Markov process is computed, with the help of the potential kernel of an additive functional associated to $D$. In \cite{Por}, the case of $\R^n$-valued symmetric stable L\'evy processes is tackled, and the law of the last exit time from a Borel set $D\subset \R^n$ is obtained in term of the equilibrium measure of $D$ (see also Takeuchi \cite{Tak} for a study of moments).\\

In the set-up of diffusions (which is our concern),  one of the main result is due to Pitman  \& Yor \cite{PY}, who compute the law of $G_f$ when $f(t)=a$ is a constant boundary and $(X_t, t\geq0)$ is a transient diffusion going to $+\infty$ a.s. (Another proof is given in \cite[Chapter 2, p.38]{PRY} using a formula for the last passage time of a continuous local martingale to a constant boundary.) We shall recover their result in Example \ref{ex:PY} below.\\

On the other hand, much emphasis has been put, quite naturally, on the study of the first passage time of $X$ to the boundary  $f$:
$$
\begin{array}{rl}
T_f&:=\inf\{t\geq0\,;\, X_t=f(t)\},\\
\vspace{-.2cm}\\
&\!(= +\infty \quad \text{if }\{t\geq0\,;\, X_t=f(t)\}=\varnothing).
\end{array}
$$
Many ideas have emerged to solve this problem: changes of measure and partial differential equations \cite{Sal,Gro}, integral equations \cite{RSS,BNR,GNRS}, martingale methods \cite{RS,Nov}, inversion of time \cite{PY,AP}...   We shall see that some of these methods may be applied to get information on last passage times.\\

\subsection{Notations}

Let $(X_t, t\geq0)$ be a linear regular diffusion taking values in $I=(\ell,+\infty[$ with $+\infty$ a natural boundary.  We assume that $(X_t, t\geq0)$ is conservative, in the sense that it has an infinite life-time, but we make no assumption on the nature of the boundary point $\ell$.
Let $\Pb_x$ and $\E_x$ denote, respectively, the probability measure and the expectation associated with $X$ when started from $x\geq \ell$. We assume that $X$ is defined on the canonical space $\Omega:=\mathcal{C}(\R^+\rightarrow I)$ and we denote by $(\F_t, t\geq0)$ its natural filtration.\\

We denote by  $m(dx)=\rho(x)dx$ its speed measure, which is assumed to be absolutely continuous with respect to the Lebesgue measure and by $s$ its scale function, which we assume to be of $\C^2$ class. 
With these notations, the infinitesimal generator of $(X_t, t\geq0)$ reads:
$$\G =\frac{\partial^2 }{\rho(x) \partial x \, s^\prime(x) \partial x} \qquad \text{for } (t,x)\in \R^+\times I.$$ 
\noindent
It is known from It\^o \& McKean \cite[p.149]{IMK} that $(X_t, t\geq0)$ admits a transition density $q(t,x,y)$ with respect to $m$, which is 
 jointly continuous and symmetric in $x$ and $y$, that is: $q(t,x,y)=q(t,y,x)$. In the remainder of the paper, we shall always assume that $q$ is a $\C^{1,2,2}$ class function on $]0,+\infty[\times I\times I$.  In particular, for any $y\in]\ell,+\infty[$, the function 
 $(t,x)\longmapsto q(t,x,y)$ is a solution of the partial differential equation:
 \begin{equation}\label{eq:q}
 \G q =\frac{\partial q}{\partial t} \qquad \text{on } ]0,+\infty[\times ]\ell, +\infty[.
 \end{equation}

\subsection{Organization of the paper}

\begin{enumerate}[$\bullet$]
\item We start in Section \ref{sec2} by giving,  under some regularity assumptions, a general expression for the density of the r.v. $G_f$:
\begin{equation}\label{eq:plan}
\Pb_x(G_f\in dt)=\Phi(t) q(t,x,f(t)) dt ,\qquad\quad (0<t<\zeta(f))
\end{equation}
where 
$$\Phi(t)=\frac{1}{s^\prime(f(t))} \frac{\partial }{\partial y}\Pb_y(T_{f(t+\centerdot)}=+\infty)|_{y=f(t)}.$$
Observe that the dependence on the initial state only appears through the transition density. Several examples of application of this formula are given, involving Brownian motion and Bessel processes with linear, squared root and square boundaries.\\

\item In Section \ref{sec3}, we follow Robbins \& Siegmund \cite{RS} and present a martingale method for computing the function $\Phi$ which appears in the previous formula. This gives us the law of $G_f$ for a large class of implicit boundaries $f$.\\

\item We then discuss, in Section \ref{sec4}, some relations between first hitting times and last passage times through time inversion.\\

\item Finally, for all initial states $x$ such that $\Pb_x(G_f>0)=1$ (i.e. the law of $G_f$ under $\Pb_x$ has no atoms), Formula (\ref{eq:plan}) implies that:
$$\int_0^{\zeta(f)}\Phi(t) q(t,x,f(t)) dt=1.$$
We shall see in Section \ref{sec5} that in some cases, this relation characterizes uniquely the function $\Phi$, hence the law of $G_f$.
\end{enumerate}

\section{The density of $G_f$}\label{sec2}

We start by establishing a general formula for the density of the last passage time $G_f$. Let $(\theta_t, t\geq0)$ denote the translation operator defined by:
$$\forall u\geq0,\qquad \left(f\circ \theta_t\right)(u)=f(t+u).$$
In the following, we shall distinguish two cases, depending on whether the process $(X_t, t\geq0)$ remains above the boundary $f$ after the last passage time $G_f$ (we shall say that $f$ is a lower boundary) or under the boundary $f$ (resp. $f$ is an upper boundary).

\subsection{Lower boundaries}

Let $f:[0,+\infty[\longrightarrow [\ell,+\infty[$ be a continuous function, which is of $\C^1$ class on $]0,+\infty[$  and define $\zeta(f):=\inf\{t\geq0\,;\, f(t)=\ell\}\in]0,+\infty]$.
If $\zeta(f)=+\infty$, we suppose (to ensure that $G_f\neq +\infty$ a.s.) that:
\begin{equation}\label{eq:up}
\forall x\in I,\qquad \Pb_x\left(\lim_{t\rightarrow \zeta(f)} X_t-f(t)>0\right)=1.
\end{equation}
This implies in particular that, for any value of $\zeta(f)\in]0,+\infty]$, the diffusion $(X_t, t\geq0)$ remains above the boundary $f$ on the time interval $]G_f, \zeta(f)[$.\\

\vspace*{.5cm}
\def\JPicScale{0.8}
\ifx\JPicScale\undefined\def\JPicScale{1}\fi
\unitlength \JPicScale mm
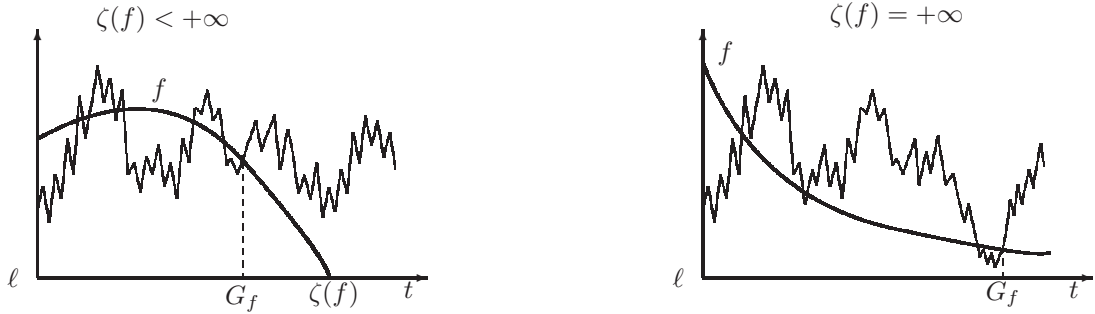
\begin{picture}(180.88,48.12)(0,0)
\linethickness{0.3mm}
\put(7.5,4.62){\line(0,1){41}}
\put(7.5,45.62){\vector(0,1){0.12}}
\linethickness{0.3mm}
\put(7.5,4.62){\line(1,0){64}}
\put(71.5,4.62){\vector(1,0){0.12}}
\linethickness{0.3mm}
\qbezier(7.5,27.62)(28.22,39.31)(41.62,23.62)
\qbezier(41.62,23.62)(55.03,7.93)(55.5,4.62)
\linethickness{0.2mm}
\multiput(7.5,15.75)(0.12,0.5){8}{\line(0,1){0.5}}
\multiput(8.5,19.75)(0.12,-0.75){8}{\line(0,-1){0.75}}
\multiput(9.5,13.75)(0.12,1){8}{\line(0,1){1}}
\multiput(10.5,21.75)(0.12,-0.5){8}{\line(0,-1){0.5}}
\multiput(11.5,17.75)(0.12,1.25){8}{\line(0,1){1.25}}
\multiput(12.5,27.75)(0.12,-0.75){8}{\line(0,-1){0.75}}
\multiput(13.5,21.75)(0.12,1.62){8}{\line(0,1){1.62}}
\multiput(14.5,34.75)(0.12,-0.88){8}{\line(0,-1){0.88}}
\multiput(15.5,27.75)(0.12,0.75){8}{\line(0,1){0.75}}
\multiput(16.5,33.75)(0.12,0.75){8}{\line(0,1){0.75}}
\multiput(17.5,39.75)(0.12,-0.75){8}{\line(0,-1){0.75}}
\multiput(18.5,33.75)(0.12,0.5){8}{\line(0,1){0.5}}
\multiput(19.5,37.75)(0.12,-0.88){8}{\line(0,-1){0.88}}
\multiput(20.5,30.75)(0.12,0.62){8}{\line(0,1){0.62}}
\multiput(21.5,35.75)(0.12,-1.75){8}{\line(0,-1){1.75}}
\multiput(22.5,21.75)(0.12,0.25){8}{\line(0,1){0.25}}
\multiput(23.5,23.75)(0.12,-0.62){8}{\line(0,-1){0.62}}
\multiput(24.5,18.75)(0.12,0.75){8}{\line(0,1){0.75}}
\multiput(25.5,24.75)(0.12,-0.38){8}{\line(0,-1){0.38}}
\multiput(26.5,21.75)(0.12,0.62){8}{\line(0,1){0.62}}
\multiput(27.5,26.75)(0.12,-0.88){8}{\line(0,-1){0.88}}
\multiput(28.5,19.75)(0.12,0.5){8}{\line(0,1){0.5}}
\multiput(29.5,23.75)(0.12,-0.75){8}{\line(0,-1){0.75}}
\multiput(30.5,17.75)(0.12,1.25){8}{\line(0,1){1.25}}
\multiput(31.5,27.75)(0.12,-0.5){8}{\line(0,-1){0.5}}
\multiput(32.5,23.75)(0.12,1.12){8}{\line(0,1){1.12}}
\multiput(33.5,32.75)(0.12,-0.12){8}{\line(1,0){0.12}}
\multiput(34.5,31.75)(0.12,0.5){8}{\line(0,1){0.5}}
\multiput(35.5,35.75)(0.12,-0.62){8}{\line(0,-1){0.62}}
\multiput(36.5,30.75)(0.12,0.25){8}{\line(0,1){0.25}}
\multiput(37.5,32.75)(0.12,-0.95){10}{\line(0,-1){0.95}}
\multiput(38.75,23.25)(0.13,0.23){5}{\line(0,1){0.23}}
\multiput(39.38,24.38)(0.12,-0.29){9}{\line(0,-1){0.29}}
\multiput(40.5,21.75)(0.12,0.38){8}{\line(0,1){0.38}}
\multiput(41.5,24.75)(0.12,0.38){8}{\line(0,1){0.38}}
\linethickness{0.2mm}
\multiput(42.5,27.62)(0.12,0.38){8}{\line(0,1){0.38}}
\multiput(43.5,30.62)(0.13,-0.71){7}{\line(0,-1){0.71}}
\multiput(44.38,25.62)(0.12,0.38){5}{\line(0,1){0.38}}
\multiput(45,27.5)(0.12,0.32){13}{\line(0,1){0.32}}
\multiput(46.5,31.62)(0.12,-1.12){8}{\line(0,-1){1.12}}
\multiput(47.5,22.62)(0.12,0.66){10}{\line(0,1){0.66}}
\multiput(48.75,29.25)(0.12,-0.6){6}{\line(0,-1){0.6}}
\multiput(49.5,25.62)(0.12,-0.88){8}{\line(0,-1){0.88}}
\multiput(50.5,18.62)(0.12,0.38){8}{\line(0,1){0.38}}
\multiput(51.5,21.62)(0.12,-0.5){8}{\line(0,-1){0.5}}
\multiput(52.5,17.62)(0.12,0.75){8}{\line(0,1){0.75}}
\multiput(53.5,23.62)(0.12,-1.13){8}{\line(0,-1){1.13}}
\multiput(54.5,14.62)(0.12,0.63){8}{\line(0,1){0.63}}
\multiput(55.5,19.62)(0.12,-0.5){8}{\line(0,-1){0.5}}
\multiput(56.5,15.62)(0.12,0.88){8}{\line(0,1){0.88}}
\multiput(57.5,22.62)(0.12,-0.5){8}{\line(0,-1){0.5}}
\multiput(58.5,18.62)(0.12,1.25){8}{\line(0,1){1.25}}
\multiput(59.5,28.62)(0.12,-0.62){8}{\line(0,-1){0.62}}
\multiput(60.5,23.62)(0.12,1){8}{\line(0,1){1}}
\multiput(61.5,31.62)(0.12,-0.62){8}{\line(0,-1){0.62}}
\multiput(62.5,26.62)(0.12,0.5){8}{\line(0,1){0.5}}
\multiput(63.5,30.62)(0.12,-0.75){8}{\line(0,-1){0.75}}
\multiput(64.5,24.62)(0.12,0.5){8}{\line(0,1){0.5}}
\multiput(65.5,28.62)(0.12,-0.75){8}{\line(0,-1){0.75}}
\put(56.25,1.88){\makebox(0,0)[cc]{$\zeta(f)$}}

\put(3.5,4.62){\makebox(0,0)[cc]{$\ell$}}

\put(27.5,35.62){\makebox(0,0)[cc]{$f$}}

\put(68.5,2.62){\makebox(0,0)[cc]{$t$}}

\linethickness{0.1mm}
\multiput(41.25,4.38)(0,2.04){10}{\line(0,1){1.02}}
\put(41.25,1.25){\makebox(0,0)[cc]{$G_f$}}

\linethickness{0.3mm}
\put(116.88,4.62){\line(0,1){41}}
\put(116.88,45.62){\vector(0,1){0.12}}
\linethickness{0.3mm}
\put(116.88,4.62){\line(1,0){64}}
\put(180.88,4.62){\vector(1,0){0.12}}
\linethickness{0.3mm}
\qbezier(117,40)(125.8,17.7)(148.38,12.75)
\qbezier(148.38,12.75)(170.95,7.8)(173.75,8.75)
\linethickness{0.2mm}
\multiput(116.88,15.75)(0.12,0.5){8}{\line(0,1){0.5}}
\multiput(117.88,19.75)(0.12,-0.75){8}{\line(0,-1){0.75}}
\multiput(118.88,13.75)(0.12,1){8}{\line(0,1){1}}
\multiput(119.88,21.75)(0.12,-0.5){8}{\line(0,-1){0.5}}
\multiput(120.88,17.75)(0.12,1.25){8}{\line(0,1){1.25}}
\multiput(121.88,27.75)(0.12,-0.75){8}{\line(0,-1){0.75}}
\multiput(122.88,21.75)(0.12,1.62){8}{\line(0,1){1.62}}
\multiput(123.88,34.75)(0.12,-0.88){8}{\line(0,-1){0.88}}
\multiput(124.88,27.75)(0.12,0.75){8}{\line(0,1){0.75}}
\multiput(125.88,33.75)(0.12,0.75){8}{\line(0,1){0.75}}
\multiput(126.88,39.75)(0.12,-0.75){8}{\line(0,-1){0.75}}
\multiput(127.88,33.75)(0.12,0.5){8}{\line(0,1){0.5}}
\multiput(128.88,37.75)(0.12,-0.88){8}{\line(0,-1){0.88}}
\multiput(129.88,30.75)(0.12,0.62){8}{\line(0,1){0.62}}
\multiput(130.88,35.75)(0.12,-1.75){8}{\line(0,-1){1.75}}
\multiput(131.88,21.75)(0.12,0.25){8}{\line(0,1){0.25}}
\multiput(132.88,23.75)(0.12,-0.98){7}{\line(0,-1){0.98}}
\multiput(133.75,16.88)(0.12,0.56){10}{\line(0,1){0.56}}
\multiput(135,22.5)(0.12,-0.5){5}{\line(0,-1){0.5}}
\multiput(135.62,20)(0.13,0.68){10}{\line(0,1){0.68}}
\multiput(136.88,26.75)(0.12,-0.88){8}{\line(0,-1){0.88}}
\multiput(137.88,19.75)(0.12,0.5){8}{\line(0,1){0.5}}
\multiput(138.88,23.75)(0.12,-0.75){8}{\line(0,-1){0.75}}
\multiput(139.88,17.75)(0.12,1.25){8}{\line(0,1){1.25}}
\multiput(140.88,27.75)(0.12,-0.5){8}{\line(0,-1){0.5}}
\multiput(141.88,23.75)(0.12,1.12){8}{\line(0,1){1.12}}
\multiput(142.88,32.75)(0.12,-0.12){8}{\line(1,0){0.12}}
\multiput(143.88,31.75)(0.12,0.5){8}{\line(0,1){0.5}}
\multiput(144.88,35.75)(0.12,-0.62){8}{\line(0,-1){0.62}}
\multiput(145.88,30.75)(0.12,0.25){8}{\line(0,1){0.25}}
\multiput(146.88,32.75)(0.12,-0.95){10}{\line(0,-1){0.95}}
\multiput(148.12,23.25)(0.13,0.23){5}{\line(0,1){0.23}}
\multiput(148.75,24.38)(0.13,-0.29){9}{\line(0,-1){0.29}}
\multiput(149.88,21.75)(0.12,0.38){8}{\line(0,1){0.38}}
\multiput(150.88,24.75)(0.12,0.38){8}{\line(0,1){0.38}}
\linethickness{0.2mm}
\multiput(151.88,27.62)(0.12,-1.52){5}{\line(0,-1){1.52}}
\multiput(152.5,20)(0.12,0.56){10}{\line(0,1){0.56}}
\multiput(153.75,25.62)(0.13,-0.75){5}{\line(0,-1){0.75}}
\multiput(154.38,21.88)(0.12,0.62){10}{\line(0,1){0.62}}
\multiput(155.62,28.12)(0.13,-0.81){10}{\line(0,-1){0.81}}
\multiput(156.88,20)(0.12,0.19){10}{\line(0,1){0.19}}
\multiput(158.12,21.88)(0.13,-0.63){5}{\line(0,-1){0.63}}
\multiput(158.75,18.75)(0.12,-0.5){10}{\line(0,-1){0.5}}
\multiput(160,13.75)(0.12,0.63){5}{\line(0,1){0.63}}
\multiput(160.62,16.88)(0.12,-0.55){16}{\line(0,-1){0.55}}
\multiput(162.5,8.12)(0.12,0.25){5}{\line(0,1){0.25}}
\multiput(163.12,9.38)(0.13,-0.5){5}{\line(0,-1){0.5}}
\multiput(163.75,6.88)(0.13,0.37){5}{\line(0,1){0.37}}
\multiput(164.38,8.75)(0.12,-0.5){5}{\line(0,-1){0.5}}
\multiput(165,6.25)(0.12,0.31){10}{\line(0,1){0.31}}
\multiput(166.25,9.38)(0.13,0.62){5}{\line(0,1){0.62}}
\multiput(166.88,12.5)(0.12,1){5}{\line(0,1){1}}
\multiput(167.5,17.5)(0.12,-0.5){5}{\line(0,-1){0.5}}
\multiput(168.12,15)(0.13,1){5}{\line(0,1){1}}
\multiput(168.75,20)(0.13,-0.62){5}{\line(0,-1){0.62}}
\multiput(169.38,16.88)(0.12,0.56){10}{\line(0,1){0.56}}
\multiput(170.62,22.5)(0.13,-0.62){5}{\line(0,-1){0.62}}
\multiput(171.25,19.38)(0.12,0.75){10}{\line(0,1){0.75}}
\multiput(172.5,26.88)(0.12,-0.75){5}{\line(0,-1){0.75}}
\put(148.75,48.12){\makebox(0,0)[cc]{$\zeta(f)=+\infty$}}

\put(112.88,4.62){\makebox(0,0)[cc]{$\ell$}}

\put(120.62,41.88){\makebox(0,0)[cc]{$f$}}

\put(177.88,2.62){\makebox(0,0)[cc]{$t$}}

\linethickness{0.1mm}
\multiput(166.25,4.38)(0,2.25){3}{\line(0,1){1.12}}
\put(166.25,1.88){\makebox(0,0)[cc]{$G_f$}}

\put(28.12,47.5){\makebox(0,0)[cc]{$\zeta(f)<+\infty$}}

\put(150.62,47.5){\makebox(0,0)[cc]{}}

\end{picture}

\vspace*{.5cm}

\noindent
Let $H$ be the function defined by:
 $$H:(t,y) \longmapsto \Pb_y(T_{f\circ \theta_t}=+\infty).$$

\noindent
We start with a general formula which gives the density of the r.v. $G_f$.

\begin{theorem}\label{theo:low}\ \\
Assume that $H$ is of $\C^{1,2}$ class on $]0,\zeta(f)[\times]\ell,+\infty[$ and is such that, for every $x\in ]\ell,+\infty[$ and $t>0$:
 \begin{equation}\label{eq:limH}
 \lim_{y\rightarrow+\infty}  \frac{\partial H( t,y)}{s^\prime(y)\partial y} q( t,x,y) =0.
 \end{equation}
Then, the density of the r.v.  $G_f$ under $\Pb_x$ is given by:
\begin{equation}\label{theo:lawGf}
\Pb_x(G_f\in dt)=\frac{q(t,x,f(t))}{s^\prime(f(t))} \; \frac{\partial H(t,y)}{\partial y}|_{y=f(t)}\; dt \quad\qquad(0< t< \zeta(f)).
\end{equation}
Note that for $\displaystyle x>f(0)$, this density is defective.
\end{theorem}

\noindent

\begin{proof}\ \\
From Assumption (\ref{eq:up}), it is clear that $G_f<\zeta(f)$ a.s. Let $0<t<\zeta(f)$:
\begin{align}
\notag \Pb_x\left(G_f\leq  t\right)&=\Pb_x\left(G_f\leq  t \cap X_ t\geq f( t)\right)\\
\notag &= \int_{f( t)}^{+\infty} \Pb_x\left(G_f\leq  t | X_ t=y\right) q( t,x,y)m(dy)\\
\notag &=\int_{f( t)}^{+\infty} \Pb_y\left(T_{f\circ\theta_ t}=+\infty\right) q( t,x,y)m(dy)\qquad (\text{from the Markov property})\\
\label{eq:intGf} &=\int_{f( t)}^{+\infty} H(t,y) q( t,x,y)m(dy).
\end{align}

\noindent
Observe that, still from the Markov property:
\begin{equation}\label{eq:Mark}
\Pb_y\left(T_f=+\infty|\F_t\right)=\Pb_{X_t}\left(T_{f\circ\theta_t}=+\infty\right)1_{\{T_f>t\}} =H(t\wedge T_f,X_{t\wedge T_f}). 
\end{equation}
Set $\displaystyle (Z_t=s(X_{t\wedge T_\ell}), t<\zeta(f))$. By construction of the scale function $s$, the process $Z$ is a continuous local martingale.  But, for $X_0>f(0)$, $\displaystyle T_f\wedge \zeta(f) \leq T_\ell\wedge \zeta(f)$, hence the process $\displaystyle (X_{t\wedge T_f}=s^{-1}(Z_{t\wedge T_f}), t<\zeta(f))$ is a semimartingale.
Therefore, applying It\^o's formula to (\ref{eq:Mark}), we deduce that the term with finite variation vanishes, i.e. $H$  is a solution to the partial differential equation:
\begin{equation}\label{eq:GH}
\G H+\frac{\partial H}{\partial t}=0 \qquad \text{ on the domain } \{(t,y);\; y> f(t)\}.
\end{equation}

\noindent
Now, let us differentiate (\ref{eq:intGf}) with respect to $t$:
\begin{equation}
\label{eq:dGdt}\Pb_x(G_f\in d t)= \int_{f( t)}^{+\infty} \frac{\partial H( t,y)}{\partial  t} q( t,x,y)m(dy) +  \int_{f( t)}^{+\infty} H( t,y) \frac{\partial q( t,x,y) }{\partial  t}  m(dy)
\end{equation}
since $H( t,f( t))=0$. Integrating by part the second integral, we obtain:
\begin{align*}
\int_{f( t)}^{+\infty} H( t,y) \frac{\partial q( t,x,y) }{\partial  t}  m(dy)&=\int_{f( t)}^{+\infty} H( t,y) \G q( t,x,y)  m(dy)\qquad \text{(from (\ref{eq:q}))}\\
&=\int_{f( t)}^{+\infty} H( t,y)  \frac{\partial^2  q( t,x,y)}{\partial y s^\prime(y) \partial y} dy\\
&=\left[H( t,y)  \frac{\partial  q( t,x,y)}{ s^\prime(y) \partial y}   \right]_{f( t)}^{+\infty} - \int_{f( t)}^{+\infty} \frac{\partial H( t,y)}{\partial y}  \frac{\partial  q( t,x,y)}{ s^\prime(y) \partial y} dy\\
&= - \int_{f( t)}^{+\infty} \frac{\partial H( t,y)}{s^\prime(y)\partial y}  \frac{\partial  q( t,x,y)}{  \partial y} dy\\
&(\text{since $H$ is bounded and $+\infty$ is a natural boundary, see \cite[p.20]{BS}})\\
&=-\left[\frac{\partial H( t,y)}{s^\prime(y)\partial y} q( t,x,y)  \right]_{f( t)}^{+\infty} + \int_{f( t)}^{+\infty} \frac{\partial^2  H( t,y)}{\partial y s^\prime(y) \partial y} q( t,x,y) dy\\
&=\frac{q( t,x,f( t))}{s^\prime(f( t))}\frac{\partial H( t,y)}{\partial y}|_{y=f( t)}   + \int_{f( t)}^{+\infty} \G H( t,y)  q( t,x,y) m(dy). 
\end{align*}
It only remains to plug this relation in (\ref{eq:dGdt}) and to use (\ref{eq:GH}) to get the desired result.\\
\end{proof}

\begin{example}[Transient diffusions and constant boundaries]\label{ex:PY}\ \\
Let $(X_t, t\geq0)$ be a transient diffusion whose scale function is of $\C^2$ class, with the normalization $s(+\infty)=0$.
Let $a>\ell$ and choose $f(t)=a$. Then, for $a<y<b$, we have, by definition of the scale function $s$:
$$\Pb_y\left(T_b<T_a\right) =\frac{s(y)-s(a)}{s(b)-s(a)}  $$
and, letting $b\rightarrow+\infty$, since $+\infty$ is a natural boundary:
 $$\Pb_y\left(T_a=+\infty\right) =1-\frac{s(y)}{s(a)} $$
so we recover the well-known formula:
$$\Pb_x(G_a\in dt)=-\frac{q(t,x,a)}{s(a)}dt, \qquad (0<t<+\infty)$$
see Pitman \& Yor \cite{PY}.  Note that in this case, Equation (\ref{eq:limH}) reduces to  $\displaystyle \lim_{y\rightarrow+\infty} q( t,x,y) =0$, which is of course always satisfied in our set-up.
\end{example}

\begin{remark}\label{rem:low} More generally, it can be proven that, for monotone functions $f$, Equation (\ref{eq:limH}) is automatically satisfied, provided that $H$ is smooth enough.  
\begin{enumerate}[$\bullet$]
\item Assume first that $f$ is increasing. Let $t$ be fixed and set $a=f(t)$. Then, for $y\geq a$, by  the continuity of paths and the strong Markov property:
\begin{align*}
\Pb_y(T_a=+\infty) &= \Pb_y(T_{f\circ\theta_t}=+\infty)+\Pb_y(T_a=+\infty \cap T_{f\circ\theta_t}<+\infty)\\
&= H(t,y) + \int_0^{+\infty}  \Pb_y(T_a=+\infty|T_{f\circ\theta_t}=s)\Pb_y(T_{f\circ\theta_t}\in ds)\\
&=H(t,y)  + \int_0^{+\infty}  \Pb_{f(t+s)}(T_a=+\infty)\Pb_y(T_{f\circ\theta_t}\in ds)\\
&\quad(\text{see Peskir \cite{Pes}})\\
&=H(t,y)  +\E_y[\Psi(T_{f\circ\theta_t})]\quad \text{where }\Psi(s)=\Pb_{f(t+s)}(T_a=+\infty).
\end{align*}
Observe now that, since $\Psi$ is increasing, all three functions are increasing functions of $y$. Therefore, we deduce, with the same normalization as in the previous example ($f$ is increasing hence $X$ must be transient), that:
$$0\leq \frac{\partial H(t,y) }{s^\prime(y)\partial y} \leq \frac{\partial \Pb_y(T_a=+\infty) }{s^\prime(y)\partial y}=-\frac{1}{s(a)},$$
which implies, since $\displaystyle \lim_{y\rightarrow+\infty} q(t,x,y)=0$, that Equation  (\ref{eq:limH}) is satisfied.\\ 
\item Assume now that $f$ is decreasing. In this case, the function $t\longmapsto H(t,y)$ is increasing, hence, from Equation (\ref{eq:GH}):
$$\frac{\partial^2 H(t,y)}{\partial y s^\prime(y)\partial y} = -\rho(y) \frac{\partial H(t,y)}{\partial t},$$
we deduce that the function $\displaystyle y\longmapsto\frac{\partial H(t,y)}{ s^\prime(y)\partial y}$ is a positive and decreasing function. 
Therefore it is bounded at $+\infty$ and Equation  (\ref{eq:limH}) is satisfied.
\end{enumerate}
\end{remark}

We now give a few examples involving Brownian motion and Bessel processes. From Theorem \ref{theo:up}, we only need to compute $\Pb_y(T_{f\circ\theta_t}=+\infty)$ to obtain the density of the last passage time to the boundary $f$.

\begin{example}[Bessel processes and straight lines: $f(t)=a-bt$]\ \\
Let $(R_t, t\geq0)$ be a Bessel process of index $\nu>-1$. $(R_t, t\geq0)$ is a diffusion whose generator reads:
$$\G^{(\nu)} = \frac{1}{2} \frac{\partial^2 }{\partial x^2} + \frac{2\nu+1}{2x} \frac{\partial}{\partial x}.$$
Its speed measure $m^{(\nu)}$ and scale function $s^{(\nu)}$ are given by:
$$
\begin{cases}
m^{(\nu)}(dx)=2x^{2\nu+1} dx,\\
(s^{(\nu)})^\prime(x)=x^{-2\nu-1}.
\end{cases}
$$
We denote by $\Pb_x^{(\nu)}$ the law of $(R_t, t\geq0)$ when started at $x$ and by $q^{(\nu)}$ its transition density function  with respect to $m^{(\nu)}$:
$$q^{(\nu)}(t,x,y)=\frac{1}{2t} (xy)^{-\nu} \exp\left(-\frac{x^2+y^2}{2t}\right) I_\nu\left(\frac{xy}{t}\right) $$
with $I_\nu$ the modified Bessel function of the third kind.\\
Choose $f(t)=a-bt$ with $a,b>0$, and thus $\displaystyle \zeta(f)=\frac{b}{a}$. Letting $\lambda$ tend toward 0 in Theorem 5.1 of  Alili \& Patie \cite{AP}, we deduce that, for $y\geq a$:
$$\Pb_y^{(\nu)}(T_f<+\infty)=\exp\left(\frac{b}{2a}(a^2-y^2)\right) \frac{y^{-\nu}}{a^{-\nu}}\int_0^{+\infty} \frac{K_\nu(\sqrt{2}yz)}{K_\nu(\sqrt{2}az)} q^{(\nu)}\left(\frac{b}{2a},\frac{b}{\sqrt{2}},z\right)m^{(\nu)}(dz)$$
where $K_\nu$ denotes the McDonald function of index $\nu$.
Then, since:
$$\Pb_y^{(\nu)}(T_{f\circ \theta_t}<+\infty)=\Pb_y^{(\nu)}(T_{a-bt-b\centerdot}<+\infty)$$
we obtain :
$$\frac{\partial H(t,y)}{\partial y}|_{y=a-bt}=b+\int_0^{+\infty}\!\!  \sqrt{2}z \frac{K_{\nu+1}}{K_\nu}\left(\sqrt{2}(a-bt)z \right)   q^{(\nu)}\left(\frac{b}{2(a-bt)},\frac{b}{\sqrt{2}},z\right)m^{(\nu)}(dz)$$
and finally, for $\displaystyle 0<  t<\frac{a}{b}$:
$$\Pb_x^{(\nu)}(G_{a-b\centerdot}\in dt)= \frac{1}{2t}(a-bt)^{\nu+1} x^{-\nu} \exp\left(-\frac{x^2+(a-bt)^2}{2t}\right) I_\nu\left(\frac{x(a-bt)}{t}  \right)\frac{\partial H(t,y)}{\partial y}|_{y=a-bt} \,dt.$$
\end{example}

\begin{example}[Reflected Brownian motion and straight lines:  $f(t)=a-bt$]\ \\
In particular, when $\nu=-\frac{1}{2}$ (i.e. when  $(X_t, t\geq0)$ is a Brownian motion reflected at 0), then $ K_{-\frac{1}{2}}=K_{\frac{1}{2}}$ and  the above formula simplifies to: 
\begin{multline*}
\Pb_x^{(-\frac{1}{2})}(G_{a-b\centerdot}\in dt)\\
=\left(b+\sqrt{\frac{b}{2\pi(a-bt)}}e^{-b(a-bt)} + b \text{Erf}\left(\sqrt{\frac{b(a-bt)}{2}} \right) \right) \left( e^{-\frac{(x+a-bt)^2}{2t}}+e^{-\frac{(x-a+bt)^2}{2t}}   \right)\frac{dt}{2\sqrt{2\pi t}}
\end{multline*}
where Erf denotes the error function: $\displaystyle \text{Erf}(z)=\frac{2}{\sqrt{\pi}}\int_0^ze^{-x^2}dx$. 
\end{example}

\begin{example}[Bessel processes and squared root boundaries: $\displaystyle f(t)=a\sqrt{1+2\gamma t}$ ]\label{exa:ROU}\ \\
Let $(R_t, t\geq0)$ be a Bessel process of index $\nu>-1$ and choose $\displaystyle f(t)=a\sqrt{1+2\gamma t}$ with $a>0$, $\gamma<0$, and thus $\displaystyle \zeta(f)=-\frac{1}{2\gamma}.$
Define 
$$\tau(t)=\frac{e^{2\gamma t}-1}{2\gamma}.$$
Then, it is known that the process
$\displaystyle  \left(e^{-\gamma t}R_{\tau(t)}, t\geq0\right)$
has the same law as a radial Ornstein-Uhlenbeck process with parameters $\nu$ and $\gamma$, i.e. 
$$(X_t=e^{-\gamma t}R_{\tau(t)}, t\geq0)$$ is a linear diffusion with characteristics:
$$
\begin{cases}
\displaystyle m^{(\nu,\gamma)}(dx) = 2 x^{2\nu+1} e^{-\gamma x^2}dx,\\
\displaystyle s^{(\nu,\gamma)}(x)=-\int_x^{+\infty} y^{-2\nu-1} e^{\gamma y^2}dy.
\end{cases}
$$
Now, we obtain by time change, since $\tau$ is increasing:
\begin{align*}
 \inf\left\{u\geq0,\; R_u=a\sqrt{1+2\gamma u}\right\} &= \tau\left(\inf\left\{t\geq0,\; R_{\tau(t)}=ae^{\gamma t}\right\}\right)\\
 &=\tau\left(\inf\left\{t\geq0,\; X_t=a\right\}\right)
\end{align*}
and we deduce from Example \ref{ex:PY} that:
$$\Pb_y^{(\nu)}(T_f=+\infty)=1-\frac{s^{(\nu,\gamma)}(y)}{s^{(\nu,\gamma)}(a)}.$$
Then, since 
$$\left(f\circ\theta_t\right)(u)=a\sqrt{1+2\gamma (t+u)}=a\sqrt{1+2\gamma t} \;\sqrt{1+\frac{2\gamma}{1+2\gamma t}u},$$
we obtain:
$$\Pb_y^{(\nu)}(T_{f\circ\theta_t}=+\infty)=1-\frac{s^{(\nu, \frac{2\gamma}{1+2\gamma t})}(y)}{s^{(\nu, \frac{2\gamma}{1+2\gamma t})}\left(a\sqrt{1+2\gamma t}\right)},$$
and finally, for $\displaystyle 0<t<-\frac{1}{2\gamma}$:
$$\Pb_x^{(\nu)}\left(G_f\in dt  \right) =-\frac{\displaystyle e^{2 \gamma a^2}}{\displaystyle  s^{(\nu, \frac{2\gamma}{1+2\gamma t})}\left(a\sqrt{1+2\gamma t}\right) } q^{(\nu)}\left(t,x,a\sqrt{1+2\gamma t}\right)dt.$$ 
\end{example}

\subsection{Upper boundaries}
We now study the case when the process $(X_t, t\geq0)$ remains under the boundary $f$ after the last passage time $G_f$.
Let $\ell\in[-\infty,+\infty[$. We make the following assumption on the nature of the boundary point $\ell$ (see \cite[p.19-20]{BS}):
\begin{enumerate}[$\bullet$]
\item If $\ell=-\infty$, we assume that $\ell$ is natural.
\item If $\ell>-\infty$, we assume that $\ell$ is entrance-not-exit.
\end{enumerate}
These hypotheses ensure that for every $t>0$ and $x>\ell$:
$$\lim_{y\rightarrow \ell} \frac{\partial q(t,x,y)}{s^\prime(y)\partial y} = 0.$$

\noindent
Let $f:[0,+\infty[\longrightarrow ]\ell,+\infty[$ be a continuous function, which is of $\C^1$ class on $]0,+\infty[$, and such that $\zeta(f):=\inf\{t\geq0\,;\, f(t)=\ell\}=+\infty$.
We assume that :
\begin{equation*}
\forall x\in I,\qquad \Pb_x\left(\lim_{t\rightarrow +\infty} X_t-f(t)<0\right)=1.
\end{equation*}
This implies that after $G_f$, the diffusion $(X_t, t\geq0)$ remains under the boundary $f$.\\
\noindent
Let $H$ be the function defined by:
 $$H:(t,y) \longmapsto \Pb_y(T_{f\circ \theta_t}=+\infty).$$

\begin{theorem}\label{theo:up}\ \\
Assume that the function $H$ is of $\C^{1,2}$ class on $]0,+\infty[\times]\ell,+\infty[$ and is such that:
 \begin{equation}\label{eq:limH2}
 \lim_{y\rightarrow \ell}  \frac{\partial H( t,y)}{s^\prime(y)\partial y} q(t,x,y) =0.
 \end{equation}
Then, the density of the r.v.  $G_f$ under $\Pb_x$ is given by:
\begin{equation}\label{theo:lawGf2}
\Pb_x(G_f\in dt)=-\frac{q(t,x,f(t))}{s^\prime(f(t))} \; \frac{\partial H(t,y)}{\partial y}|_{y=f(t)}\; dt \quad\qquad(t>0).
\end{equation}
\end{theorem}

\begin{proof}\ \\
The proof is of course very similar to the previous one. For $0<t<+\infty$, we have:
$$\Pb_x\left(G_f\leq  t\right)=\int_{\ell}^{f(t)}H(t,y) q( t,x,y)m(dy)$$
with the function $H$ solution of
\begin{equation}\label{eq:GH2}
\G H+\frac{\partial H}{\partial t}=0 \qquad \text{ on the domain } \{(t,y);\; y< f(t)\}.
\end{equation}
(Note that, by hypothesis, $T_\ell=+\infty$ a.s. hence, since $s$ is a function of $\C^2$ class, $X$ is a semimartingale).
Then, we make the same integrations by parts, noticing that the terms between brackets cancel since:
\begin{enumerate}[$i)$]
\item $H$ is bounded and $\displaystyle \frac{\partial q( t,x,y)}{s^\prime(y)\partial y} \xrightarrow[y\rightarrow \ell]{} 0$,
\item  $\displaystyle  \lim_{y\rightarrow \ell}  \frac{\partial H( t,y)}{s^\prime(y)\partial y} q( t,x,y) =0$ by hypothesis.
\end{enumerate}
\end{proof}

\begin{example}[Bessel processes and straight lines: $f(t)=a+bt$]\label{ex:Ba+b}\ \\
Let $(R_t, t\geq0)$ be a Bessel process of index $\nu>0$ and choose $f(t)=a+bt$ with $a,b>0$. Then, from Theorem 5 of Alili \& Patie \cite{AP}, for $y< a$:
$$\Pb_y^{(\nu)}(T_{a+b\centerdot}\in du)=\exp\left(\frac{b}{2a}(y^2-a^2) -\frac{b^2}{2}u\right) \left(1+\frac{b}{a}u\right)^{\nu-1} \sum_{k= 1}^{+\infty}  \frac{y^{-\nu} j_{\nu,k}J_\nu(j_{\nu,k} \frac{y}{a})}{a^{2-\nu} J_{\nu+1}(j_{\nu,k})}\exp\left(-\frac{u\, j^2_{\nu,k} }{2a(a+bu)}\right)du,$$
where $J_\nu$ denotes the Bessel function of the first kind and $(j_{\nu,k})_{k\geq0}$ the ordered sequence of its positive zeroes.
 Then, 
 $$H(t,y)= 1- \Pb^{(\nu)}_y(T_{a+bt+b\centerdot}<+\infty)$$
and, for $ x\geq0$:
\begin{align*}
\Pb_x^{(\nu)}(G_{a+b\centerdot}\in dt)&= -\frac{1}{2t}(a+bt)^{\nu+1} x^{-\nu} \exp\left(-\frac{x^2+(a+bt)^2}{2t}\right) I_\nu\left(\frac{x(a+bt)}{t}  \right)\frac{\partial H(t,y)}{\partial y}|_{y=a+bt} \,dt,\\
&= -\frac{1}{2t}(a+bt)^{\nu+1} x^{-\nu} \exp\left(-\frac{x^2+(a+bt)^2}{2t}\right) I_\nu\left(\frac{x(a+bt)}{t}  \right)\left(-b+\psi^\prime(a+bt)\right) dt
\end{align*}
with
$$\psi(y)= \int_0^{+\infty}  \exp\left(-\frac{b^2}{2}u\right) \left(1+\frac{b}{a+bt}u\right)^{\nu-1} \sum_{k= 1}^{+\infty}  \frac{y^{-\nu} j_{\nu,k}J_\nu(j_{\nu,k} \frac{y}{a+bt})}{(a+bt)^{2-\nu} J_{\nu+1}(j_{\nu,k})}\exp\left(- \frac{u \,j^2_{\nu,k} }{2(a+bt)(a+b(t+u))}\right)du.$$

\end{example}

\begin{example}[Brownian motion and square boundary: $f(t)=a+bt^2$]\label{ex:Bt2}\ \\
Let $(B_t, t\geq0)$ be a Brownian motion and define $f(t)=a+bt^2$ with $a,b>0$. From Salminen \cite{Sal}, we know that:
$$\forall y<a,\qquad \Pb_y(T_f\in du)=2 (bc)^2\exp\left(-\frac{2}{3}b^2 u^3\right) \sum_{k=0}^{+\infty} \exp\left(\frac{\lambda_k}{c}u \right)\frac{\text{Ai}(\lambda_k+2bc(a-y))}{\text{Ai}^\prime(\lambda_k)}du
$$
where Ai denotes the Airy function,  $(\lambda_k)_{k\geq0}$ its zeroes on the negative half-line and $\displaystyle c=(2b^2)^{-\frac{1}{3}}$.
Now fix $t\geq0$. Applying the Cameron-Martin formula, we deduce that:
$$\Pb_{y-bt^2}^{(-2bt)}(T_f\in du)=\exp\left(-2bt(a-y+bt^2)-(2bt) b u^2 -\frac{(2bt)^2}{2}u \right)\Pb_{y-bt^2}(T_f\in du),$$
where $\Pb_x^{(\mu)}$ denotes the law of a Brownian motion with drift $\mu$ started at $x$. But, since
$$\inf\{u\geq0;\;  (y-bt^2)+B_u-2btu = a+bu^2\}=\inf\{u\geq0;\;  y+B_u= a+b(u+t)^2\},$$
we obtain that $\Pb_y(T_{f\circ\theta_t}< +\infty)=\Pb_{y-bt^2}^{(-2bt)}(T_f < +\infty)$ and
\begin{multline*}
H(t,y)=1-2 (bc)^2\exp\left(-\frac{4}{3}b^2t^3-2bt(a-y) \right)\\
\times\int_0^{+\infty}   \exp\left(-\frac{2}{3}b^2(u+t)^3  \right)  \sum_{k=0}^{+\infty} \exp\left(\frac{\lambda_k}{c}u \right)\frac{\text{Ai}(\lambda_k+2bc(a-y+bt^2))}{\text{Ai}^\prime(\lambda_k)}du.
\end{multline*}
Finally, for $x\in \R$,
\begin{align*}
\Pb_x(G_{f}\in dt)&= - \frac{1}{2\sqrt{2\pi t}}\exp\left(-\frac{(x-a-bt^2)^2}{2t}\right) \frac{\partial H(t,y)}{\partial y}|_{y=a+bt^2}\,dt,\\
&=- \frac{1}{2\sqrt{2\pi t}}\exp\left(-\frac{(x-a-bt^2)^2}{2t}\right)\left(-2bt+ \psi^\prime(a+bt^2)\right)
\end{align*}
with
$$\psi(y)= -2 (bc)^2\exp\left(\frac{2}{3}b^2t^3\right) \int_0^{+\infty}   \exp\left(-\frac{2}{3}b^2(u+t)^3  \right)  \sum_{k=0}^{+\infty} \exp\left(\frac{\lambda_k}{c}u \right)\frac{\text{Ai}(\lambda_k+2bc(a-y+bt^2))}{\text{Ai}^\prime(\lambda_k)}du.
$$
\end{example}

\section{Martingales methods}\label{sec3}
We now present a method to obtain explicit expressions for the function $H$ (associated with an a priori implicit boundary).
Let $\ell\in[-\infty,+\infty[$  and  $f:[0,+\infty[\longmapsto ]\ell,+\infty[$ be a continuous function.
In this section, we shall restrict our attention to lower boundaries, and make the following Assumption:

\begin{assumption}\label{Ass}
If $\zeta(f)=+\infty$, we assume that:
$$ \lim_{t\rightarrow +\infty}X_t=+\infty  \;\; a.s \qquad \text{ and }\quad \Pb_x\left(\lim_{t\rightarrow +\infty} X_t-f(t)>0\right)=1.$$
\end{assumption}

\noindent
Consider the domains:
$$\mathcal{D}=\{(t,y)\in]0,\zeta(f)[\times]\ell,+\infty[;\; y\geq f(t)\} \quad \text{ and } \quad \partial \mathcal{D}:=\{(t,y)\in]0,\zeta(f)[\times]\ell,+\infty[; \; y=f(t)\}.$$

\begin{lemma}\label{lem:mart}
Assume that Assumption \ref{Ass} holds and that there exists a function $\overline{H}:\mathcal{D}\longmapsto [0,1]$ of $\C^{1,2}$ class on $\mathcal{D}$ which is solution of the following problem:
\begin{equation}\label{eq:PDE}
\begin{cases}
\displaystyle \G \overline{H} +\frac{\partial \overline{H}}{\partial t} =0,\qquad \text{with the boundary condition } \;
\displaystyle \overline{H}(t,y)=1 \text{ on }\partial \mathcal{D},\\
\vspace{-.3cm}\\
\displaystyle \overline{H} \text{ is decreasing in $y$}\\
\vspace{-.3cm}\\
 \begin{array}{c|c}
 \displaystyle \text{If } \zeta(f)=+\infty,\quad&  \displaystyle \text{If } \zeta(f)<+\infty,\\
\displaystyle \lim_{t,y \rightarrow +\infty} \overline{H}(t,y)=0\quad & \displaystyle \forall y\in]\ell,+\infty[, \; \lim_{t\rightarrow \zeta(f)} \overline{H}(t,y)=0.  
 \end{array}
\end{cases}
\end{equation}
Then, $$\overline{H}(t,y)=\Pb_y(T_{f\circ\theta_t}<+\infty).$$
\end{lemma}

\begin{proof}\ \\
Let $t<\zeta(f)$ be fixed. Applying It\^o's formula, we deduce that the process $\left(\overline{H}(t+s,X_s), s< (\zeta(f)-t)\wedge T_{f\circ\theta_t}\right)$ is a positive and continuous local martingale. Let $s<\zeta(f)-t$ be fixed. From Doob' stopping theorem, with $f(t)\leq y\leq a$:
\begin{align*}
\overline{H}(t,y)&=\E_y\left[\overline{H}(t+s\wedge T_a \wedge T_{f\circ \theta_t}, X_{s\wedge T_a \wedge T_{f\circ \theta_t}})\right]\\
&=\E_y\left[\overline{H}(t+s,X_s)1_{\{s< T_a\wedge T_{f\circ \theta_t}\}}\right]+
\E_y\left[\overline{H}(t+T_a,a) 1_{\{T_a< s\wedge T_{f\circ \theta_t}\}} \right]+\Pb_y\left( T_{f\circ \theta_t} < s\wedge T_a\right).
\end{align*}
We first let $a\rightarrow+\infty$ to obtain, applying the dominated convergence theorem:
$$\overline{H}(t,y)=\E_y\left[\overline{H}(t+s,X_s)1_{\{s<  T_{f\circ \theta_t}\}}\right]+\Pb_y\left( T_{f\circ \theta_t} < s\right).$$
Then, we must distinguish between two cases.

\begin{enumerate}
\item If $\zeta(f)=+\infty$, then, we assumed that $\displaystyle \lim_{s\rightarrow +\infty}X_s=+\infty$ a.s., so:
$$\overline{H}(t+s,X_s) \xrightarrow[s\rightarrow+\infty]{} 0\quad \text{a.s.}$$
and the result follows from the dominated convergence theorem.
\item If $\zeta(f)<+\infty$, then, since $\Pb_y(X_{\zeta(f)-t}=\ell)=0$: 
$$ \lim_{s\rightarrow \zeta(f)-t} \overline{H}(t+s, X_s) =0 \quad \text{a.s.}$$
and the result follows once again from the dominated convergence theorem.
\end{enumerate}
\end{proof}

We shall now give some examples of functions $\overline{H}$ which are solutions of this problem.

\subsection{Martingales constructed on the resolvant}\label{sub31}
We follow the idea of Robbins \& Siegmund \cite{RS}.  Let us define, for $\lambda>0$, the resolvent kernel of $(X_t,t\geq0)$ by  (see \cite[p.19]{BS}):
 \begin{equation*}
u_\lambda(x,y)=\int_0^\infty e^{-\lambda t}q(t,x,y) dt. 
\end{equation*}
\noindent
Let $F$ be a finite measure on $[0,+\infty[$, and define:
$$\overline{H}(t,y)=\int_0^{+\infty} e^{-\lambda t} u_\lambda(0,y) F(d\lambda).$$
We assume furthermore that:
\begin{enumerate}
\item if $\ell=-\infty$, then $\ell$ is a natural boundary,
\item if $\ell>-\infty$, then $\ell$ is an entrance-not-exit boundary.
\end{enumerate}
This implies that $\displaystyle\lim_{y\rightarrow \ell} u_\lambda(0,y)=+\infty$.
Now, since the function $y\longmapsto u_\lambda(0,y)$ is strictly decreasing, we may define a function $f:[0,+\infty[\longrightarrow ]\ell,+\infty[$ by
$$\forall t\geq0,\qquad \overline{H}(t,f(t))=1,$$
which is such that $\zeta(f)=+\infty$. Then:

\begin{proposition}
The density of $G_f$ is given by:
$$\Pb_x(G_f\in dt)=-\frac{q(t,x,f(t))}{s^\prime(f(t))} \;\left(\int_0^{+\infty} e^{-\lambda t}  \frac{\partial  u_\lambda(0,y) }{\partial y}|_{y=f(t)}F(d\lambda) \right) dt\qquad (0<t<+\infty).$$
\end{proposition}

\begin{proof}
We need to check that the function  $\overline{H}$ satisfies the hypotheses of Lemma \ref{lem:mart}. \\
By construction, $\overline{H}$ is a solution to the partial differential equation (\ref{eq:PDE}).
It is also a decreasing function of $y$ and, since $+\infty$ is a natural boundary, we have from the monotone convergence theorem, $\displaystyle \lim_{t,y\rightarrow+\infty} \overline{H}(t,y)=0$. 

\end{proof}

\begin{example}[Brownian motion with drift]\ \\
Let $(X_t, t\geq0)$ be a Brownian motion with drift $\mu$, and choose $F(d\lambda) = 2\mu a \delta_0(d\lambda) + 2\sqrt{2b+\mu^2}\,\delta_b(d\lambda)$. Then:
$$\overline{H}(t,y)=ae^{-2\mu y} +e^{-bt-\left(\sqrt{2b+\mu^2}+\mu\right) y}.$$
Define:
$$\varphi(y)=-\frac{\left(\sqrt{2b+\mu^2}+\mu\right)y + \ln(1-ae^{-2\mu y})}{b}\qquad \text{ and}\quad  f(t)=\varphi^{-1}(t).$$
$f$ is a decreasing function such that $\displaystyle \lim_{t\rightarrow+\infty}f(t)=\frac{\ln(a)}{2\mu}$.
Then:
$$\Pb_x(G_f\in dt)=-\frac{q(t,x,f(t))}{s^\prime(f(t))}  \left(a\left(\sqrt{2b+\mu^2}-\mu\right)e^{-2\mu f(t)} -\sqrt{2b+\mu^2}-\mu  \right) dt.$$
\end{example}

\begin{example}[Bessel process of dimension 3]\ \\
Let $(X_t, t\geq0)$ be a Bessel process of dimension 3, and choose $F(d\lambda) = a \delta_0(d\lambda) + \delta_b(d\lambda)$. Then:
$$\overline{H}(t,y)=\frac{a}{y}+\frac{e^{-b t-\sqrt{2b}y}}{y}.$$
Define:
$$\varphi(y)=-\frac{\sqrt{2b}y +\ln(y-a)}{b}\qquad \text{ and}\quad f(t)=\varphi^{-1}(t).$$ 
$f$ is a decreasing function such that $\displaystyle \lim_{t\rightarrow+\infty}f(t)=a$.
Then :
$$\Pb_x(G_f\in dt)=\frac{q(t,x,f(t))}{s^\prime(f(t))}  \left(\frac{1}{f(t)}+\sqrt{2b} \left(1-\frac{a}{f(t)}\right)   \right) dt.$$

\end{example}

\subsection{Martingales constructed on the transition density}
We assume in this subsection that $(X_t, t\geq0)$  is defined on $(0,+\infty[$. 

\subsubsection{Preliminaries}

We start by recalling a few properties of the transition density $(t,x)\longmapsto q(t,x,0)$.
From Kotani \& Watanabe \cite{KW}, it is known that :
 \begin{equation}\label{eq:KW}
 \lim_{t\rightarrow 0} (-2t) \log\left(q(t,x,y)\right)=\left(\int_x^y \sqrt{\frac{\rho(z)s^\prime(z)}{2}}dz   \right)^2.
 \end{equation}
In \cite{KW}, this formula was obtained in the special case $s(x)=x$, but, assuming that $s$ is a strictly increasing function of $\C^1$ class, Formula (\ref{eq:KW}) follows easily from the fact that $s(X)$ is a diffusion on natural scale.  
In particular, for $x>0$, we deduce that:
 \begin{equation}\label{eq:KWq}
 \lim_{t\rightarrow 0}q(t,x,0)=0.
 \end{equation}
Next, we define a new diffusion $(\overline{X}_t, t\geq0)$ whose speed measure  $\overline{m}(dx)=\overline{\rho}(x)dx$ and scale function $\overline{s}$  are given by Biane's transform:
\begin{equation*}
\begin{cases}
\displaystyle \overline{\rho}(x)=(m([0,x]))^2 s^\prime(x) \\
\displaystyle \overline{s}(x)=\frac{1}{m([0,+\infty[)}-\frac{1}{m([0,x])}.
\end{cases}
\end{equation*}
It is known from \cite[Chapter 8]{PRY} that the transition densities of $X$ and $\overline{X}$ satisfy the following relation:
 \begin{equation}\label{eq:Bia}
  q(t,y,0)=\int_y^\infty \overline{q}(t,0,z) m([0,z]) s^\prime(z) dz,
  \end{equation}
which implies in particular that for every $t>0$, the function $y\longmapsto q(t,y,0)$ is decreasing and tends toward 0 as $y\longrightarrow+\infty$.

\subsubsection{First example}
Let $\zeta,c>0$ and consider the function:
$$\overline{H}(t,y)= \frac{1}{c} q(\zeta-t,y,0).$$
When $\displaystyle c<\inf_{t<\zeta}q(\zeta-t, 0,0)$ we may define a boundary $f$ by:
$$ q(\zeta-t,f(t),0)=c.$$
In this set-up, $\zeta=\inf\{t\geq0;\; f(t)=0\}=\zeta(f)$ from (\ref{eq:KWq}). Then:

\begin{proposition}
The density of $G_f$ is given by:
$$\Pb_x(G_f\in dt)=- \frac{q(t,x,f(t))}{cs^\prime(f(t))} \frac{\partial q(\zeta-t,y,0)}{\partial y} |_{y=f(t)} dt\qquad (0<t<\zeta).$$
\end{proposition}

\begin{proof}
From (\ref{eq:q}), the function $\overline{H}$ is a solution of the PDE (\ref{eq:PDE}), which is decreasing from (\ref{eq:Bia}) and, for every $y>0$, $\displaystyle \lim_{t\rightarrow \zeta}\overline{H}(t,y)=0$ from (\ref{eq:KWq}).\\
\end{proof}
Note that the martingale $(\overline{H}(t,X_t), t<\zeta)$ appears as a density when constructing diffusion bridges via Doob's $h$-transform, see Fitzsimmons, Pitman \& Yor \cite{FPY}.
\begin{example}[Radial Ornstein-Uhlenbeck processes]\ \\
Let $(X_t,t\geq0)$ be a squared radial Ornstein-Uhlenbeck process with parameters $\nu>-1$ and $\gamma<0$, see Example \ref{exa:ROU}. Its transition density function reads (see \cite[p.142]{BS}):
$$q^{(\nu,\gamma)}(t,y,0)=\frac{\gamma^{\nu+1} e^{\gamma (\nu+1)t}}{2^{\nu+1}\Gamma(\nu+1) (\sinh(\gamma t))^{\nu+1}}\exp\left(-\frac{\gamma e^{-\gamma t} y^2}{2\sinh(\gamma t)}\right). $$
Let $\alpha \in]0,1[$, $\zeta\in]0,+\infty[$ and choose:
$$c =\alpha q^{(\nu,\gamma)}(\zeta,0,0)$$
With these values,  the boundary $f^{(\nu,\gamma)}$ is then defined, for $t< \zeta$ by:
$$f^{(\nu,\gamma)}(t)=\sqrt{-\frac{2}{\gamma} \sinh(\gamma (\zeta-t))e^{\gamma (\zeta-t)} \ln\left(\alpha \left(\frac{\sinh(\gamma( \zeta-t))}{\sinh(\gamma \zeta)} \right)^{\nu+1}e^{\gamma (\nu+1)t}\right)}$$
and the density of the r.v. $G_f$ is given by:
$$\Pb_x(G_{f^{(\nu,\gamma)}}\in dt)= \frac{\gamma e^{-\gamma (\zeta-t)}}{c\sinh(\gamma (\zeta-t))} \left( f^{(\nu,\gamma)}(t) \right)^{2\nu+2} e^{-\gamma\left( f^{(\nu,\gamma)}(t)\right)^2} q^{(\nu,\gamma)}(t,x,f^{(\nu,\gamma)}(t))dt\qquad (0<t<\zeta). $$
\end{example}

\begin{example}[Bessel process]\ \\
Letting $\gamma\rightarrow 0$ in the previous example, we obtain the following boundary for the Bessel process of index $\nu$:
$$f^{(\nu)}(t)=\sqrt{-2  (\zeta-t) \ln\left(\alpha \left(1-\frac{t}{\zeta}\right)^{\nu+1}\right)},$$
and the density for the last passage time to $f^{(\nu)}$ reads:
$$\Pb_x^{(\nu)}(G_{f^{(\nu)}}\in dt)= \frac{1}{c(\zeta-t)} \left( f^{(\nu)}(t) \right)^{2\nu+2}  q^{(\nu)}(t,x,f^{(\nu)}(t))dt,\qquad (0<t<\zeta), $$
with $\displaystyle c=\frac{\alpha}{2^{\nu+1}\Gamma(\nu+1)t^{\nu+1}}$.
\end{example}

\subsubsection{Second example}
We assume that $0$ is an entrance-not-exit boundary point and that $(X_t, t\geq0)$ is transient and goes towards $+\infty$ as $t\rightarrow+\infty$. Let $h:[0,+\infty[\longrightarrow ]0,+\infty[$ be a bounded function of $\C^1$ class with bounded derivative. We consider the function: 
 $$\overline{H}(t,y)=\int_0^{+\infty}h(t+u)q(u,y,0)du.$$
 Since we have assumed that 0 is entrance-not-exit,  $u_\lambda(0,y)\xrightarrow[y\rightarrow0]{}+\infty$, i.e. $q(t,0,0)$ is not integrable at 0. Therefore, for every $t\geq0$, the function $y\longmapsto \overline{H}(t,y)$ is decreasing from $+\infty$ to $0$, and we may define a function $f$ by:
 $$\int_0^{+\infty}h(t+u)q(u,f(t),0)du=1.$$
 Note that by construction, $\zeta(f)=+\infty$.

\begin{proposition}
The density of $G_f$ is given by:
$$\Pb_x(G_f\in dt)=-\frac{q(t,x,f(t))}{s^\prime(f(t))}\left( \int_0^{+\infty}h(u+t)\frac{\partial q(u,y,0)}{\partial y} |_{y=f(t)}du  \right) dt\qquad (t>0).$$
\end{proposition}

\begin{remark}
We may recover (partly) the result of Subsection \ref{sub31}  by taking $h(u)=\int_0^{+\infty} e^{-\lambda u} F(d\lambda)$. Indeed, in this case, from Fubini-Tonelli:
$$\int_0^{+\infty}h(u+t) q(u,y,0)du = \int_0^{+\infty}   \left(\int_0^{+\infty} e^{-\lambda (t+u)} F(d\lambda)\right)  q(u,y,0)du =  \int_0^{+\infty} e^{-\lambda t} u_\lambda (y,0)F(d\lambda). $$
\end{remark}

\begin{proof}
We need to prove that $\overline{H}$ satisfies the hypotheses of Lemma \ref{lem:mart}.\\
Since $h$ is bounded, and for $y>0$, $u\longmapsto q(u,y,0)$ is integrable,  the dominated convergence theorem implies that :
$$\lim_{t,y\rightarrow +\infty} \overline{H}(t,y)=0.$$
Then, integrating by parts, for $y>0$:
\begin{align*}
\frac{\partial \overline{H}(t,y)}{\partial t} &=\int_0^{+\infty} h^\prime(t+u) q(u,y,0)du \qquad (\text{which is finite since $h^\prime$ is bounded})\\
&= \Big[h(t+u)q(u,y,0)\Big]_0^{+\infty} - \int_0^{+\infty} h(t+u) \frac{\partial q(u,y,0)}{\partial u}du\\
&=-  \left(\int_0^{+\infty} h(t+u) \G q(u,y,0)du\right)\\
&=-\G \overline{H}(t,y),
\end{align*}
which ends the proof.\\
\end{proof}

\begin{remark}
We may remove the hypothesis $(X_t, t\geq0)$ is transient if we replace the assumption on $h$ by~: $h$ is a decreasing and integrable function of $\C^1$ class. Indeed, since for $y>0$ the function $u\longmapsto q(u,y,0)$ is bounded, $\overline{H}$ is well-defined and so is $\int_0^{+\infty} h^\prime(t+u) q(u,y,0)du$. Besides, we still have $\displaystyle \lim_{t,y\rightarrow +\infty} \overline{H}(t,y)=0$ from monotone convergence.
\end{remark}

\section{Inverting time}\label{sec4}

Consider a diffusion $(X_t, t\geq0)$ enjoying the inversion property in the sense of Watanabe \cite{Wat}, i.e. such that the process $(\overline{X}_t=tX_{\frac{1}{t}}, t\geq0)$ is also a linear regular conservative diffusion.
Let $f$ be a continuous function and define:
$$\overline{f}:t\longmapsto tf\left(\frac{1}{t}\right).$$
Then, if  $\displaystyle \overline{X}_0\neq \lim_{t\rightarrow0}\overline{f}(t)$:
\begin{align*}
\inf\left\{t\geq0;\, \overline{X}_t=\overline{f}(t)\right\} &= \inf\left\{t\geq0;\, tX_{\frac{1}{t}}=tf\left(\frac{1}{t}\right)\right\}\\
&=\inf\left\{t\geq0;\, X_{\frac{1}{t}}=f\left(\frac{1}{t}\right)\right\}  \\
&= \frac{1}{\sup\{u\geq0;\, X_u=f(u)\}}.
\end{align*}
In particular, under the hypotheses of Theorem \ref{theo:low} or \ref{theo:up}, the density of the first hitting time of $\overline{f}$ by $(\overline{X}_t, t\geq0)$ admits the expression:
$$\overline{\Pb}_{\overline{x}}(T_{\overline{f}}\in dt)=  \frac{1}{t^2}\Phi\left(\frac{1}{t}\right) q\left(\frac{1}{t},x,f\left(\frac{1}{t}\right)\right)dt.$$

\subsection{Brownian motion}
Consider a Brownian motion $(B_t, t\geq0)$ started from 0. Then, it is well-known that $(\overline{B}_t, t\geq0)$ is also a Brownian motion started at 0.

\begin{example}
Take $f(u)=a+b u^2$ with $a,b>0$. Then, $\displaystyle \overline{f}(t)=at+\frac{b}{t}$ and :
$$\inf\left\{t\geq0;\; \overline{B}_t=at+\frac{b}{t}\right\} = \frac{1}{\sup\{u\geq0;\; B_u=a+bu^2\}}$$
From Example \ref{ex:Bt2}, we deduce that:
$$\overline{\Pb}_0(T_{\overline{f}}\in dt)=  \Phi\left(\frac{1}{t}\right) \frac{1}{2t\sqrt{2\pi t}} \exp\left(-\frac{1}{2t}\left(at+\frac{b}{t}\right)^2   \right)dt$$
where $\Phi$ is defined by :
$$\Phi\left(\frac{1}{t}\right) = \frac{2b}{t} - \psi^\prime\left(a+\frac{b}{t^2}\right)$$
with
$$\psi(y)=-2 (bc)^2\exp\left(\frac{2}{3}\frac{b^2}{t^3}\right)
\int_0^{+\infty}   \exp\left(-\frac{2}{3}b^2\left(u+\frac{1}{t}\right)^3  \right)  \sum_{k=0}^{+\infty} \exp\left(\frac{\lambda_k}{c}u \right)\frac{\text{Ai}\left(\lambda_k+2bc\left(a-y+\frac{b}{t^2}\right)\right)}{\text{Ai}^\prime(\lambda_k)}du.
$$
with $(\lambda_k)_{k \geq0}$ the negative zeroes of the Airy function Ai and $c=(2b^2)^{-1/3}$.

\end{example}

\subsection{Bessel processes with drift}
Consider a Bessel process $(R_t, t\geq0)$ with index $\nu>0$ and drift $c\geq0$ started from $x$. $(R_t, t\geq0)$ is a diffusion whose generator is given by:
$$\G^{(\nu,c)}= \frac{1}{2}\frac{\partial^2}{\partial x^2} +\left(\frac{2\nu+1}{2x}+c \frac{I_{\nu+1}(c x)}{I_\nu(c x)} \right)\frac{\partial}{\partial x}.$$
This process was first introduced by Watanabe in \cite{Wat} as a generalization of Bessel processes. For integer dimension of $n=2(\nu+1)\in \N$, $(R_t, t\geq0)$ has the same law as the norm $\| \vec{B_t}+ \vec{\mu} \centerdot \vec{ t}\|$ where $(\vec{B_t}, t\geq0)$ is an $n$-dimensional Brownian motion and $\| \vec{ \mu}\|= c$.\\
The scale function of $R$ is given by:
$$(s^{(\nu,c)})^\prime(x)=\frac{1}{(\Gamma(\nu+1))^2 }\left(\frac{c}{2}\right)^{2\nu}\frac{1}{x I_\nu^2(c x)}$$
and its speed measure by:
$$m^{(\nu,c)}(dx)= 2 (\Gamma(\nu+1))^2\left(\frac{2}{c}\right)^{2\nu}x I_\nu^2(c x)dx.$$
From Watanabe \cite[Theorem 2.1]{Wat}, the process $\displaystyle \left(\overline{R}_t=tR_{\frac{1}{t}}, t\geq0\right)$ is a Bessel process with index $\nu$ and drift $x$ started from $c$.
\begin{example}
 Take $f(u)=a+bu$ with $a,b>0$. Then $\overline{f}(t)=b+at$ and 
$$\inf\{t\geq0;\, \overline{R}_t=b+at\} =  \frac{1}{\sup\{u\geq0;\, R_u=a+bu\} }$$
which leads to the density of the r.v. $T_{b+a\centerdot}$ under the form:
$$\Pb_0^{(\nu,x)}(T_{b+a\centerdot}\in dt)=\frac{1}{2t} \Phi\left(\frac{1}{t}\right) \left(x\left(a+\frac{b}{t}\right)\right)^{-\nu} \exp\left(-\frac{x^2t^2+(at+b)^2}{2t}\right) I_\nu\left(x(a t+b)\right)dt $$
where the function $\Phi$ is given by (see Example \ref{ex:Ba+b}): 
$$\Phi\left(\frac{1}{t}\right) = \left(a+\frac{b}{t}\right)^{2\nu+1}\left(b-\psi^\prime\left(a+\frac{b}{t}\right)\right)
$$
with
$$\psi(y)=\int_0^{+\infty}  \left(1+\frac{bt}{at+b}u\right)^{\nu-1} \sum_{k= 1}^{+\infty}  t^2\frac{(yt)^{-\nu} j_{\nu,k}J_\nu(j_{\nu,k} \frac{yt}{at+b})}{(at+b)^{2-\nu} J_{\nu+1}(j_{\nu,k})}\exp\left(-\frac{u\,t^2\, j^2_{\nu,k} }{2(at+b)(at+b(1+tu))}-\frac{b^2}{2}u\right)du,
$$
where $J_\nu$ denotes the Bessel function of the first kind and $(j_{\nu,k})_{k\geq0}$ the ordered sequence of its positive zeroes.
 \end{example}

\section{On an integral equation}\label{sec5}

\noindent
We set:
$$\Phi(t)=\frac{1}{s^\prime(f(t))}\frac{\partial }{\partial y}\Pb_y\left(T_{f\circ\theta_t}=+\infty\right)|_{y=f(t)}.$$
Integrating (\ref{theo:lawGf}) and (\ref{theo:lawGf2}) with respect to $t$, we deduce the following corollary:

\begin{corollary}\label{cor:phi}
Assume that the hypotheses of Theorem \ref{theo:low}, resp. \ref{theo:up}, are satisfied. Then, the function $\Phi$  is a solution of the following Fredholm equation of the first kind:
\begin{enumerate}[$\bullet$]
\item For lower boundaries:
\begin{equation*}
\int_0^{\zeta(f)}\Phi(t) q(t,x,f(t)) dt=1, \qquad  x< f(0).
\end{equation*}
\item Resp, for upper boundaries:
\begin{equation*}
\int_0^{+\infty}\Phi(t) q(t,x,f(t)) dt=1, \qquad  x> f(0).
\end{equation*}
\end{enumerate}
\end{corollary}
In some particular cases, this equation may charaterize uniquely $\Phi$, hence the law of $G_f$. 
We give below a few examples of this situation, where time inversion is involved.

\subsection{A link with time inversion}

\begin{theorem}\label{theo:inteq}\ \\
\vspace*{-.6cm}
\begin{enumerate}[$i)$]
\item Let $(B_t^{(\mu)}, t\geq0)$ be a Brownian motion with drift $\mu>0$ started from $x$ and  $f$ be a continuous function on $[0,+\infty[$ such that 
$$f(0)=0 \qquad \text{and}\qquad \lim_{t\rightarrow +\infty} \frac{f(t)}{ t}<\mu.$$ 
Assume that the equation :
\begin{equation*}
\forall x< 0,\qquad \int_0^{+\infty}\Phi(t) q(t,x,f(t)) dt=1,
\end{equation*}
admits a unique solution $\Phi$. Then :
$$\Pb_x^{(\mu)}\left(G_f\in dt\right) =\Phi(t)q(t,x,f(t))dt.$$
\item Let $(R_t, t\geq0)$ be a Bessel process with index $\nu>-1$ and drift $c\geq0$ started from $x>0$, and let $f$ be a continuous function on $[0,+\infty[$ such that:
$$f(0)=0 \qquad \text{and}\qquad \lim_{t\rightarrow +\infty} \frac{f(t)}{t}>c.$$ 
Assume that the equation
\begin{equation*}
\forall x>0,\qquad \int_0^{+\infty}\Phi(t) q^{(\nu,c)}(t,x,f(t)) dt=1,
\end{equation*}
admits a unique solution $\Phi$. Then for $x>0$:
$$\Pb_x^{(\nu,c)}\left(G_f\in dt\right) =\Phi(t)q^{(\nu,c)}(t,x,f(t))dt.$$
\end{enumerate}
\end{theorem}

\begin{proof}\ \\
$i)$ The transition density of $(B_t^{(\mu)}, t\geq0)$ reads:
$$q(t,z,y)=\frac{1}{2\sqrt{2\pi t}} \exp\left(-\mu(z+y) -\frac{\mu^2}{2}t - \frac{(z-y)^2}{2t}\right).$$
Observe first that:
$$q(t,z,y)=q(t,0,y)   \exp\left(-\frac{z^2}{2t} + \frac{z y}{t}-\mu z\right),$$
hence, by hypothesis, $\Phi$ is the unique solution of the equation:
$$\forall z< 0,\qquad \int_0^{+\infty}\Phi(t) q(t,0,f(t))\exp\left(-\frac{z^2}{2t} + \frac{zf(t)}{t}\right)dt=e^{\mu z},
$$
or, with the change of variable $s=\frac{1}{t}$:
$$\forall z< 0,\qquad \int_0^{+\infty}\Phi\left(\frac{1}{s}\right) q\left(\frac{1}{s},0,f\left(\frac{1}{s}\right)\right)\exp\left(-\frac{z^2}{2}s + zs f\left(\frac{1}{s}\right)\right)\frac{ds}{s^2}=e^{\mu z}.
$$
Now, let $\displaystyle \overline{f}(s)=sf\left(\frac{1}{s}\right)$ and consider, for $z<0$, the exponential Brownian martingale $\left(\exp\left(z B_s-\frac{z^2}{2} s\right), s\geq0\right)$. Doob' stopping theorem implies, since $\displaystyle \mu > \lim_{s\rightarrow 0}\overline{f}(s)$ by hypothesis:
$$e^{z\mu }=\E_\mu\left[ \exp\left(z B_{T_{\overline{f}}}  -\frac{z^2}{2}T_{\overline{f}}\right)  \right]=\int_0^{+\infty} e^{z \overline{f}(s) -\frac{z^2}{2}s}\, \Pb_{\mu}\left(T_{\overline{f}}\in ds \right).$$
Since $\overline{f}$ is a positive function, $T_{\overline{f}}$ admits a density and by unicity of the solution $\Phi$, we have:
$$ \Pb_{\mu}\left(T_{\overline{f}}\in ds \right) =\frac{1}{s^2} \Phi\left(\frac{1}{s}\right) q\left(\frac{1}{s},0,f\left(\frac{1}{s}\right)\right) ds.$$
Next, the Cameron-Martin formula gives:
$$\Pb_\mu^{(x)}\left(T_{\overline{f}}<s\right)=\E_\mu\left[ \exp\left(x B_{T_{\overline{f}}}  -\frac{x^2}{2}T_{\overline{f}}-\mu x\right)1_{\{T_{\overline{f}}<s\}}  \right],$$
i.e. the density of the r.v. $T_{\overline{f}}$ equals:
 $$\Pb_{\mu}^{(x)}\left(T_{\overline{f}}\in ds \right) = \exp\left(x \overline{f}(s)  -\frac{x^2}{2}s-\mu x\right) \frac{1}{s^2} \Phi\left(\frac{1}{s}\right) q\left(\frac{1}{s},0,f\left(\frac{1}{s}\right)\right) ds=\frac{1}{s^2} \Phi\left(\frac{1}{s}\right) q\left(\frac{1}{s},x,f\left(\frac{1}{s}\right)\right) ds.$$
The result then follows from the time inversion property.\\

\noindent
$ii)$ We shall proceed similarly for Bessel processes with drift. The transition density of $(R_t, t\geq0)$ reads:
$$q^{(\nu,c)}(t,z,y)=\frac{1}{2t(\Gamma(\nu+1))^2}\left(\frac{c}{2}\right)^{2\nu} \frac{\exp\left(-\frac{c^2t}{2}\right)}{I_\nu(cz)I_\nu(cy)}\exp\left(-\frac{z^2+y^2}{2t}\right) I_\nu\left(\frac{zy}{t}\right),$$
hence, let $\Phi$ be the unique solution of the equation:
$$\forall z>0,\qquad \int_0^{+\infty}  \frac{\Phi(t)}{2t(\Gamma(\nu+1))^2}\left(\frac{c}{2}\right)^{2\nu} \frac{\exp\left(-\frac{c^2t}{2}-\frac{(f(t))^2}{2t}\right)}{I_\nu(cf(t))}\exp\left(-\frac{z^2}{2t}\right) \frac{I_\nu\left(\frac{zf(t)}{t}\right)}{I_\nu(cz)}dt=1.$$
Consider the local martingale under $\Pb_{c}^{(\nu,x)}$: 
$$\left(M_t=e^{-\lambda t}\frac{I_\nu\left(X_t \sqrt{2\lambda+x^2}\right)}{I_\nu(x X_t)}, t\geq0\right).$$
Applying Doob' stopping theorem to $(M_{T_{\overline{f}}\wedge T_\varepsilon\wedge t}, t\geq0)$ and letting $t\rightarrow +\infty$ and $\varepsilon\rightarrow 0$, we obtain (since $\displaystyle c<\lim_{t\rightarrow 0}\overline{f}(t)$ and $\overline{f}(t)\mathop{=}\limits_{+\infty}o(t)$ by hypothesis):
$$\frac{I_\nu\left(c \sqrt{2\lambda+x^2}\right)}{I_\nu(x c)}=\E_c^{(\nu,x)}\left[ M_{T_{\overline{f}}}\right]=\int_0^{+\infty} e^{-\lambda t}  \frac{I_\nu\left(\overline{f}(t) \sqrt{2\lambda+x^2}\right)}{I_\nu(x  \overline{f}(t) )}\Pb_c^{(\nu,x)}(T_{\overline{f}}\in dt). $$
Now, from Watanabe \cite{Wat}:
$$\Pb_c^{(\nu,x)}\left(\lim_{t\rightarrow+\infty} \frac{X_t}{t}=x  \right)=1,$$
hence, since  $\overline{f}(t)\mathop{=}\limits_{+\infty}o(t)$, we deduce that, for $x>0$, the r.v. $T_{\overline{f}}$ admits a density (see also  \cite{Leh}) and, setting  $2\lambda+x^2=z^2$, we obtain from the unicity of the solution $\Phi$:
$$\frac{I_\nu\left(cx\right)}{I_\nu(x\overline{f}(t))}e^{\frac{x^2}{2}t}\Pb_c^{(\nu,x)}\left(T_{\overline{f}}\in dt\right) =\frac{1}{t^2}  \frac{t \Phi\left(\frac{1}{t}\right)}{2(\Gamma(\nu+1))^2}\left(\frac{c}{2}\right)^{2\nu} \frac{\exp\left(-\frac{c^2}{2t}-\frac{t}{2}\left(f\left(\frac{1}{t}\right)\right)^2\right)}{I_\nu\left(cf\left(\frac{1}{t}\right)\right)}dt.$$
The result then follows once again from a time inversion argument. \\
Observe that, if $\Pb_c^{(\nu,0)}(T_{\overline{f}}<+\infty)=1$, (for instance if $f$ is of $\C^1$ class in the neighborhood of 0), then the result also holds for $x=0$.\\
\end{proof}

\noindent
It might be noticed that for these two processes, the proof above shows anew the phenomenon of separation of variables which appears in the law of $G_f$.
We shall now give an example of each situation.

\subsection{Brownian motion with drift and $f(t)=a+b\sqrt{t}$}
Let $(B_t^{(\mu)}, t\geq0)$ be a Brownian motion with drift $\mu>0$ and choose $f(t)=a+b\sqrt{t}$ with $a,b\in\R$.
\begin{proposition}
The density of the last passage time $G_{a+b\sqrt{\centerdot}}$ is given, for every $x\in \R$, by:
$$\Pb_x^{(\mu)}\left(G_{a+b\sqrt{\centerdot}}\in dt\right)=\varphi(t) \exp\left(-\frac{(x-a)^2}{2t} +\frac{(x-a)b}{2\sqrt{t}}-\mu(x-a)-\frac{b^2}{4} \right) dt$$
where the function $\varphi$ has Mellin's transform:
$$\int_0^{+\infty} t^{\lambda-1}\varphi(t)dt= \frac{ 1}{\mu^{2\lambda-1} D_{-2\lambda+1}\left(b\right)}.$$
\noindent
In particular, if $x=a$, we deduce that:
$$\E_a\left[ G^{\lambda-1}_{a+b\sqrt{\centerdot}}\right]=\frac{ \exp\left(-\frac{b^2}{4}\right)}{\mu^{2\lambda-1} D_{-2\lambda+1}\left(b\right)}.$$
\end{proposition}

\begin{proof}
We need to prove that the equation:
\begin{equation}\label{eq:PhiDnu}
\forall z>0,\qquad \int_0^{+\infty}   \Phi(t)q(t,0,b\sqrt{t})  \exp\left(-\frac{z^2}{2t} - \frac{zb}{\sqrt{t}}\right)dt=\exp\left(-\mu z  +2\mu a\right),
\end{equation}
admits a unique solution $\Phi$, in order to apply Theorem \ref{theo:inteq}. Let us recall the following formula (\cite[3.462-1]{GR}):
$$\int_0^{+\infty} z^{\nu-1} \exp\left(-\frac{\beta}{2} z^2-\gamma z\right)dz = \frac{\Gamma(\nu)}{(\beta)^{\nu/2}}  \exp\left(\frac{\gamma^2}{4\beta}\right) D_{-\nu}\left(\frac{\gamma}{\sqrt{\beta}}\right)\qquad \text{with }\beta,\nu>0, \text{ and } \gamma\in \R,$$
where $D_{-\nu}$ denotes the parabolic cylinder function of index $-\nu$.
Integrating Equation (\ref{eq:PhiDnu}) with respect to $z^{\nu-1} dz$ and applying Fubini-Tonelli, we obtain:
$$\int_0^{+\infty} \Phi(t) q(t,0,b\sqrt{t}) \Gamma(\nu) t^{\nu/2}  \exp\left(\frac{b^2}{4}\right) D_{-\nu}\left(b\right)dt= \frac{\Gamma(\nu)}{\mu^\nu} \exp\left(2\mu a\right)  $$
which, by setting $\lambda-1=\frac{\nu}{2}$, gives the Mellin's transform:
$$\int_0^{+\infty} \Phi(t) q(t,0,b\sqrt{t})t^{\lambda-1} dt = \frac{ \exp\left(2\mu a-\frac{b^2}{4}\right)}{\mu^{2\lambda-1} D_{-2\lambda+1}\left(b\right)}.$$
Since Mellin's transform is injective, the result follows from Theorem \ref{theo:inteq}.\\
\end{proof}

\subsection{Bessel process with drift and $f(t)=\sqrt{at^2+bt}$}\label{sub42}
Let $(R_t, t\geq0)$ be a Bessel process with index $\nu>-1$ and drift $c\geq0$ started from $x$, and choose $f(t)=\sqrt{at^2+bt}$ with $b>0$ and $\sqrt{a}>c$.
\begin{proposition}
The density of the last passage time $G_{\sqrt{a \centerdot^2+b\centerdot}}$ is given, for every $x\geq0$, by:
$$\Pb_x^{(\nu,c)}\left(G_{\sqrt{a \centerdot^2+b\centerdot}}\in dt\right)=\varphi\left(\ln\left(1+\frac{b}{at}\right)\right) \frac{b}{ct}\frac{1}{\sqrt{at^2+bt}} \frac{ \exp\left(-\frac{x^2}{2t}\right)  }{I_\nu(cx)} I_\nu\left(\frac{x\sqrt{at^2+bt}}{t}\right)dt$$
where the function $\varphi$ has Laplace transform:
\begin{equation}\label{eq:Lphi}
\int_0^{+\infty} e^{-\lambda t}\varphi(t)dt= \frac{\exp\left(\frac{c^2b}{4a}\right)M_{-\lambda,\nu/2}\left(\frac{c^2b}{2a}\right)}{\exp\left(\frac{b}{4}\right)M_{-\lambda,\nu/2}\left(\frac{b}{2}\right)}.
\end{equation}
\end{proposition}

\begin{proof}
We need to prove that the equation :
\begin{equation}\label{eq:PhiInu}
\forall x>0,\qquad \int_0^{+\infty} \frac{\Phi(t)}{2t(\Gamma(\nu+1))^2}\left(\frac{c}{2}\right)^{2\nu}  \frac{e^{-\frac{c^2t}{2}}}{I_\nu(cf(t))}\exp\left(-\frac{x^2+(f(t))^2}{2t}\right) I_\nu\left(\frac{x\sqrt{at^2+bt}}{t}\right)dt=I_\nu(cx),
\end{equation}
admits a unique solution $\Phi$ in order to apply Theorem \ref{theo:inteq}. To simplify the notation, set 
$$\displaystyle \Psi(t)=\frac{\Phi(t)}{2t(\Gamma(\nu+1))^2}\left(\frac{c}{2}\right)^{2\nu}\frac{e^{-\frac{c^2t}{2}}}{I_\nu(cf(t))}\exp\left(-\frac{(f(t))^2}{2t}\right).$$
Recall the formula \cite[6.643-2]{GR}, for $\alpha>0, \beta\in \R$ and $\lambda+\nu +\frac{1}{2}>0$:
$$\int_0^{+\infty} x^{2\lambda} e^{-\alpha x^2} I_{2\nu}(2\beta x) dx =\frac{\Gamma(\lambda+\nu+\frac{1}{2})}{2\Gamma(2\nu+1)} \frac{1}{\beta\alpha^\lambda} \exp\left(\frac{\beta^2}{2\alpha}\right) M_{-\lambda,\nu}\left(\frac{\beta^2}{\alpha}\right)$$ 
where $M_{-\lambda,\nu}$ denotes the Whittaker function.
We integrate (\ref{eq:PhiInu}) with respect to $\displaystyle x^{2\lambda} e^{-\frac{a}{2b}x^2} dx$, to obtain:
$$\int_0^{+\infty}  \Psi(t)  \frac{2t}{\sqrt{at^2+bt}}\frac{e^{\frac{b}{4}}}{(\frac{a}{2b}+\frac{1}{2t})^{\lambda}}  M_{-\lambda,\nu/2}\left(\frac{b}{2}\right)dt= \frac{2}{c} \left(\frac{2b}{a}\right)^\lambda \exp\left(\frac{c^2b}{4a}\right)M_{-\lambda,\nu/2}\left(\frac{c^2b}{2a}\right).$$ 
This expression simplifies to:
$$\int_0^{+\infty}  \Psi(t)  \sqrt{\frac{t}{at+b}}\frac{1}{(1+\frac{b}{at})^{\lambda}}dt= \frac{1}{c}  \exp\left(\frac{c^2b}{4a}-\frac{b}{4}\right)\frac{M_{-\lambda,\nu/2}\left(\frac{c^2b}{2a}\right)}{M_{-\lambda,\nu/2}\left(\frac{b}{2}\right)}.$$ 
Finally, we make the change of variable $\displaystyle e^u=1+\frac{b}{at}$:
$$\int_0^{+\infty} e^{-\lambda u} \Psi\left(\frac{b}{a(e^u-1)}  \right) \frac{e^{\frac{u}{2}}}{(e^u-1)^2}du= \frac{a\sqrt{a}}{bc}\exp\left(\frac{c^2b}{4a}-\frac{b}{4}\right)\frac{M_{-\lambda,\nu/2}\left(\frac{c^2b}{2a}\right)}{M_{-\lambda,\nu/2}\left(\frac{b}{2}\right)}$$ 
which gives the Laplace transform of $\Psi$ up to a few transformations. \\
\end{proof}

\begin{remark}
Taking $\lambda=\alpha-\frac{\nu+1}{2}$, we obtain, since $\sqrt{a}>c$:
$$
\int_0^{+\infty} e^{-\alpha t} e^{\frac{\nu+1}{2}t}\varphi(t)dt= \frac{\exp\left(\frac{c^2b}{4a}\right)M_{-\alpha + \frac{\nu+1}{2},\nu/2}\left(\frac{c^2b}{2a}\right)}{\exp\left(\frac{b}{4}\right)M_{-\alpha + \frac{\nu+1}{2},\nu/2}\left(\frac{b}{2}\right)}=\left(\frac{a}{c^2}\right)^{\frac{\nu+1}{2}} \Q^{(\nu,\frac{1}{2})}_{c\sqrt{\frac{b}{a}}}\left[e^{-\alpha T_{\sqrt{b}}}\right]
$$
where $ \Q_{x}^{(\nu, \gamma)}$ denotes the law of a radial Ornstein-Uhlenbeck process with parameters $\nu$ and $\gamma$ started at $x$. Therefore:
$$\varphi(t)dt = \exp\left(-\frac{\nu+1}{2}t\right)\left(\frac{a}{c^2}\right)^{\frac{\nu+1}{2}} \Q^{(\nu,\frac{1}{2})}_{c\sqrt{\frac{b}{a}}} \left(T_{\sqrt{b}} \in dt\right).$$
\end{remark}

\begin{corollary}
Take $f(u)=\sqrt{au^2+bu}$ with $b>0$ and $\sqrt{a}>c$.
 Then  $\overline{f}(t)=\sqrt{a+bt}$ and:
$$\inf\{t\geq0;\, \overline{R}_t=\sqrt{a+bt}\} =  \frac{1}{\sup\{u\geq0;\, R_u=\sqrt{au^2+bu}\} }$$
which leads to the density of the r.v. $T_{\sqrt{a+b\centerdot}}$ under the form:
$$\Pb_{c}^{(\nu,x)}(T_{\sqrt{a+b\centerdot}}\in dt)=\varphi\left(\ln\left(1+\frac{bt}{a}\right)\right) \frac{b}{c}\frac{1}{\sqrt{a+bt}} \frac{ \exp\left(-\frac{x^2}{2}t\right)  }{I_\nu(cx)} I_\nu\left(x\sqrt{a+bt}\right)dt$$
where the function $\varphi$ is given by (\ref{eq:Lphi}).
\end{corollary}

\bibliographystyle{alpha}

\nocite{*}
\begin{thebibliography}{GNRS89}

\bibitem[AP10]{AP}
L.~Alili and P.~Patie.
\newblock Boundary crossing identities for diffusions having the time-inversion
  property.
\newblock {\em J. Theoret. Probab.}, 23(1):65--84, 2010.

\bibitem[BNR87]{BNR}
A.~Buonocore, A.~G. Nobile, and L.~M. Ricciardi.
\newblock A new integral equation for the evaluation of first-passage-time
  probability densities.
\newblock {\em Adv. in Appl. Probab.}, 19(4):784--800, 1987.

\bibitem[BS02]{BS}
A.~N. Borodin and P.~Salminen.
\newblock {\em Handbook of {B}rownian motion---facts and formulae}.
\newblock Probability and its Applications. Birkh\"auser Verlag, Basel, second
  edition, 2002.

\bibitem[FPY93]{FPY}
P.~Fitzsimmons, J.~Pitman, and M.~Yor.
\newblock Markovian bridges: construction, {P}alm interpretation, and splicing.
\newblock In {\em Seminar on {S}tochastic {P}rocesses, 1992 ({S}eattle, {WA},
  1992)}, volume~33 of {\em Progr. Probab.}, pages 101--134. Birkh\"auser
  Boston, Boston, MA, 1993.

\bibitem[GNRS89]{GNRS}
V.~Giorno, A.~G. Nobile, L.~M. Ricciardi, and S.~Sato.
\newblock On the evaluation of first-passage-time probability densities via
  non-singular integral equations.
\newblock {\em Adv. in Appl. Probab.}, 21(1):20--36, 1989.

\bibitem[GR07]{GR}
I.~S. Gradshteyn and I.~M. Ryzhik.
\newblock {\em Table of integrals, series, and products}.
\newblock Elsevier/Academic Press, Amsterdam, seventh edition, 2007.

\bibitem[Gro89]{Gro}
P.~Groeneboom.
\newblock Brownian motion with a parabolic drift and {A}iry functions.
\newblock {\em Probab. Theory Related Fields}, 81(1):79--109, 1989.

\bibitem[GS73]{GS}
R.~K. Getoor and M.~J. Sharpe.
\newblock Last exit times and additive functionals.
\newblock {\em Ann. Probability}, 1:550--569, 1973.

\bibitem[IM74]{IMK}
K.~It{\^o} and H.~P. McKean.
\newblock {\em Diffusion processes and their sample paths}.
\newblock Springer-Verlag, Berlin, 1974.
\newblock Second printing, corrected, Die Grundlehren der mathematischen
  Wissenschaften, Band 125.

\bibitem[KW82]{KW}
S.~Kotani and S.~Watanabe.
\newblock Kre\u\i n's spectral theory of strings and generalized diffusion
  processes.
\newblock In {\em Functional analysis in {M}arkov processes ({K}atata/{K}yoto,
  1981)}, volume 923 of {\em Lecture Notes in Math.}, pages 235--259. Springer,
  Berlin, 1982.

\bibitem[Leh02]{Leh}
A.~Lehmann.
\newblock Smoothness of first passage time distributions and a new integral
  equation for the first passage time density of continuous {M}arkov processes.
\newblock {\em Adv. in Appl. Probab.}, 34(4):869--887, 2002.

\bibitem[Nov81]{Nov}
A.~A. Novikov.
\newblock A martingale approach to first passage problems and a new condition
  for {W}ald's identity.
\newblock In {\em Stochastic differential systems ({V}isegr\'ad, 1980)},
  volume~36 of {\em Lecture Notes in Control and Information Sci.}, pages
  146--156. Springer, Berlin, 1981.

\bibitem[Pes02]{Pes}
G.~Peskir.
\newblock On integral equations arising in the first-passage problem for
  {B}rownian motion.
\newblock {\em J. Integral Equations Appl.}, 14(4):397--423, 2002.

\bibitem[Por67]{Por}
S.~C. Port.
\newblock Hitting times for transient stable processes.
\newblock {\em Pacific J. Math.}, 21:161--165, 1967.

\bibitem[PRY10]{PRY}
C.~Profeta, B.~Roynette, and M.~Yor.
\newblock {\em Option prices as probabilities}.
\newblock Springer Finance. Springer-Verlag, Berlin, 2010.
\newblock A new look at generalized Black-Scholes formulae.

\bibitem[PY81]{PY}
J.~Pitman and M.~Yor.
\newblock Bessel processes and infinitely divisible laws.
\newblock In {\em Stochastic integrals ({P}roc. {S}ympos., {U}niv. {D}urham,
  {D}urham, 1980)}, volume 851 of {\em Lecture Notes in Math.}, pages 285--370.
  Springer, Berlin, 1981.

\bibitem[RS70]{RS}
H.~Robbins and D.~Siegmund.
\newblock Boundary crossing probabilities for the {W}iener process and sample
  sums.
\newblock {\em Ann. Math. Statist.}, 41:1410--1429, 1970.

\bibitem[RSS84]{RSS}
L.~M. Ricciardi, L.~Sacerdote, and S.~Sato.
\newblock On an integral equation for first-passage-time probability densities.
\newblock {\em J. Appl. Probab.}, 21(2):302--314, 1984.

\bibitem[Sal88]{Sal}
P.~Salminen.
\newblock On the first hitting time and the last exit time for a {B}rownian
  motion to/from a moving boundary.
\newblock {\em Adv. in Appl. Probab.}, 20(2):411--426, 1988.

\bibitem[Tak67]{Tak}
J.~Takeuchi.
\newblock Moments of the last exit times.
\newblock {\em Proc. Japan Acad.}, 43:355--360, 1967.

\bibitem[Wat75]{Wat}
S.~Watanabe.
\newblock On time inversion of one-dimensional diffusion processes.
\newblock {\em Z. Wahrscheinlichkeitstheorie und Verw. Gebiete}, 31:115--124,
  1974/75.

\end{thebibliography}

\end{document}